\documentclass[11pt]{article}

\usepackage{amsmath}
\usepackage{amsthm}
\usepackage{amsfonts}
\usepackage{tikz-cd}
\usepackage{hyperref}

\theoremstyle{plain}
\newtheorem{Theorem}{Theorem}[section]
\newtheorem{Proposition}[Theorem]{Proposition}
\newtheorem{Lemma}[Theorem]{Lemma}
\newtheorem{Corollary}[Theorem]{Corollary}

\newtheorem{Conjecture}[Theorem]{Conjecture}

\theoremstyle{definition}
\newtheorem{Definition}[Theorem]{Definition}

\newtheorem{Remark}[Theorem]{Remark}

\newcommand{\cat}[1]{\text{\rm cat}\left(#1 \right)}
\newcommand{\TC}[1]{\text{\rm TC}\left(#1 \right)}
\newcommand{\higherTC}[2]{\text{\rm TC}_{#1} \left(#2 \right)}
\newcommand{\secat}[1]{\text{\rm secat}\left(#1 \right)}
\newcommand{\B}{$\bullet$}

\title{Surgery Applications to a Generalized Rudyak Conjecture}
\author{Jamie Scott}

\begin{document}
\maketitle

\begin{abstract}
Rudyak's conjecture states that $\cat{M} \geq \cat{N}$ given a degree one map $f:M \to N$ between closed manifolds. We generalize this conjecture to sectional category, and follow the methodology of \cite{Dranishnikov_Scott_2020} to get the following result:
\begin{Theorem}
Consider the following commutative diagram:
\[
\begin{tikzcd}
E^M \arrow[r, "\overline{f}"] \arrow[d, swap, "p^M"]
& E^N \arrow[d, "p^N"]
\\
M \arrow[r, "f"]
& N
\end{tikzcd}
\]
where $p^M$ and $p^N$ are fibrations, and $f$ is a normal map of degree $1$ between closed smooth manifolds. Let the fiber $F^N$ of $p^N$ be $(r-2)$-connected for some $r \geq 1$. If $f$ has zero surgery obstructions and $N$ satisfies $5 \leq \dim N \leq 2r \secat{p^N} - 3$, then $\secat{p^M} \geq \secat{p^N}$.
\end{Theorem}
Finally, we apply this result to the case of higher topological complexity when $N$ is simply connected.
\end{abstract}

\section{Introduction}

The (reduced) {\it Lusternik-Schnirelmann category}, $\cat X$, of an ANR space $X$ is the minimal number $k$ such that $X$ admits an open cover by $k+1$ sets $U_0,\dots,U_k$ that each admit a nulhomotopy of their inclusion map $U_i \hookrightarrow X$. Classically, this invariant is used as a lower bound for the number of critical points for smooth real-valued functions \cite{Lusternik_Schnirelmann_1929}. Rudyak's conjecture is the following statement:
\begin{Conjecture}[Rudyak's Conjecture]
If $f:M\to N$ is a degree one map between closed manifolds, then
$\cat M \geq \cat N$.
\end{Conjecture}
This conjecture is motivated by the following two points:
\begin{enumerate}
    \item[\B] In some sense, the LS-category can be seen as a measure of how complicated the space $X$ is.
    \item[\B] A degree one map between closed manifolds induces a split surjection (injection) on homology (cohomology) \cite{Dranishnikov_Sadykov_2019, Rudyak_1999}.
\end{enumerate}
It's currently unknown whether this conjecture is true in general, but various special cases have been proven (e.g., see results in \cite{Rudyak_1999,Rudyak_2017}), the most recent of which is the following theorem:
\begin{Theorem}[\cite{Dranishnikov_Scott_2020}]\label{thm:surgery_Rudyak}
Let $f:M \to N$ be a normal map of degree one between closed smooth manifolds 
with $N$ being $(r-1)$-connected, $r\ge 1$. If $N$ satisfies the inequality
$\dim N \leq 2r \cat N - 3$, 
then {\rm $\cat M \geq \cat N$}.
\end{Theorem}
In a recent paper \cite{Rudyak_Sarkar_2020}, a new version of Rudyak's conjecture was considered in the context of topological complexity, but the present work first considers an even more general version of the conjecture for sectional category. Given a fibration $p:E \to B$, the {\it sectional category} of $p$, denoted $\secat{p}$, is the smallest integer $k$ such that $B$ can be covered by $k+1$ open sets $U_0,\dots,U_k$ that each admit a partial section $s_i:U_i \to E$ of $p$. So given a homotopy commutative diagram of the form
\[
\begin{tikzcd}
E^M \arrow[r, "\overline{f}"] \arrow[d, swap, "p^M"]
& E^N \arrow[d, "p^N"]
\\
M \arrow[r, "f"]
& N
\end{tikzcd}
\]
where $f:M \to N$ is a degree one map, we ask when the inequality $\secat{p^M} \geq \secat{p^N}$ is satisfied. In section 3, we prove that if $f$ has no surgery obstructions, then we get a theorem nearly identical to Theorem \ref{thm:surgery_Rudyak}. Unfortunately, unlike in \cite{Dranishnikov_Scott_2020}, we are unable to remove the assumption that there are no surgery obstructions. In \cite{Dranishnikov_Scott_2020}, the surgery obstruction of $f$ was removed by taking connected sums in light of the formula $\cat{M_1 \# M_2} = \max \{ \cat{M_1}, \cat{M_2} \}$ (see \cite{Dranishnikov_Sadykov_2019,Dranishnikov_Sadykov_2020}). In the general case, there are two problems with this approach: 1) in the case of  general fibrations it is unclear how to define the connected sum of fibrations and 2) when the connected sum does make sense, the formula is not necessarily known, e.g., the formula for the topological complexity of $M_1 \# M_2$ is not in general known in terms of the topological complexities of $M_1$ and $M_2$.

In section 4, we consider the case of higher topological complexity and restrict ourselves to when $N$ is simply-connected. When $\dim M \not\equiv 0 (\text{mod }4)$, we prove that the surgery obstruction of $f^{\times k}$ is trivial for all $k \geq 2$ and we obtain the desired result:

\begin{Theorem}
Let $f:M \to N$ be a normal map of degree $1$ between closed smooth manifolds such that $N$ is $(r-1)$-connected for some $r \geq 2$ (i.e., $N$ is simply connected). Suppose that $N$ satisfies the inequality $5 \leq k \dim N \leq 2r \higherTC{k}{N} - 3$. If $\dim N \not\equiv 0 (\text{\rm mod} 4)$, then $\higherTC{k}{M} \geq \higherTC{k}{N}$.
\end{Theorem}

Finally, we restrict ourselves to the case of $k=2$ in order to study the case when $\dim M \equiv 0 (\text{mod }4)$, and using what little is known about the topological complexity of a connected sum, we get the following somewhat less satisfying result:

\begin{Theorem}
Let $f:M \to N$ be a normal map of degree one between $(r-1)$-connected ($r \geq 2$) smooth closed $4n$-manifolds, and assume that $M$ satisfies the inequality $\frac{4n}{r} + 2 \leq \TC{M}$. Then $\TC{M} \geq \TC{N}$.
\end{Theorem}

In section 4, the above theorem is stated assuming a weaker inequality.

\section{Preliminaries}

\subsection{Fiberwise Joins}

For $1 \leq i \leq n$, let $p_i:E_i \to B$ be a fibration with fiber $F_i$. The fiberwise join of the total spaces $E_i$ over the base $B$ is defined to be the subspace of the usual join $E_1 \ast ... \ast E_n$ given by set
\[ E_1 \ast_B ... \ast_B E_n
= \left\{ t_1e_1 + ... + t_ne_n \in E_1 \ast ... \ast E_n \mid t_i,t_j \neq 0 \implies p_i(e_i) = p_j(e_j) \right\}.\]
Then the fiberwise join of the fibrations $p_i$ is the map
\[ p_1 \ast_B ... \ast_B p_n: E_1 \ast_B ... \ast_B E_n \to B\]
given by
\[ t_1e_1 + ... + t_ne_n \mapsto p_i(e_i)\]
for any $i$ such that $t_i \neq 0$. This map is a well defined fibration by the definition of the space $E_1 \ast_B ... \ast_B E_n$.
This operation is called the fiberwise join as the fiber of $p_1 \ast_B ... \ast_B p_n$ is the join of the fibers $F_1 \ast ... \ast F_n$.

If $p_i=p:E \to B$ for all $i$, then we denote the fibration $p_1 \ast_B ... \ast_B p_n$ by $\ast_B^n p$ and we denote its total space $E_1 \ast_B ... \ast_B E_n$ by $\ast_B^n E$.

The following theorem connects the fiberwise joins with calculation of the sectional category:

\begin{Theorem} [\cite{Schwarz_1966}] \label{thm:Ganea-Schwarz_approach}
Let $p:E \to B$ be a fibration with $B$ normal. Then $\secat{p} \leq n$ if and only if $\ast_B^{n+1} p$ admits a section.
\end{Theorem}

\subsection{Higher Topological Complexity}

Let $X$ denote the configuration space of a mechanical system, then a {\it motion planner} of the system is a map $\sigma:X\times X\to X^I$, $I=[0,1]$, to the path space such that $\sigma(x_0,x_1)(0)=x_0$ and $\sigma(x_0,x_1)(1)=x_1$. Clearly, if there exists a continuous motion planner, then $X$ is contractible. Since configuration spaces are rarely contractible, one has to partition $X\times X$ into  pieces such that on each piece there is a continuous motion planner. M. Farber called the minimal number of such pieces \cite{Farber_2003} (also see \cite{Farber_2008}) the {\em topological complexity} of $X$
and denoted it by $\TC{X}$. Or, more formally:

\begin{Definition}
Let $X$ be a path-connected topological space. The (reduced) {\it topological complexity} of $X$, denoted $\TC{X}$, is the least $k$ such that $X \times X$ can be covered by $k+1$ open subsets $U_0,...,U_k$ on which there are continuous motion planners. If no such $k$ exists, we define $\TC{X} = \infty$.
\end{Definition}

More generally, one could define a notion of topological complexity for motion planning for paths with $k-2$ intermediary stops. This is called {\it higher topological complexity} and was initially introduced in \cite{Rudyak_2010}. Again, more formally:

\begin{Definition}
Let $X$ be a path-connected topological space, let $k \geq 2$, and let $\Delta^X_k: X^I \to X^k$ be the path space fibration given by \[\Delta^X_k(\alpha) = (x_1,...,x_k).\]
Then the (reduced) {\it $k$-th topological complexity} of $X$ is $\higherTC{k}{X} := \secat{\Delta^X_k}$.
\end{Definition}

Note that $\TC{X} = \higherTC{2}{X}$ by definition, so that higher TC is indeed an extension of topological complexity.

Not much is known about the higher topological complexity of the connected sum of two manifolds in general, but we will use the following two known results in Section 4:

\begin{Theorem}[\cite{Dranishnikov_Sadykov_2019}]
\label{thm:connected_sum_inequality}
Let $M_1$ and $M_2$ be closed $(r-1)$-connected $n$-manifolds such that $\TC{M_i} \geq \frac{n+2}{r}$ for one of $i=1$ or $i=2$. Then
\[ \TC{M_1 \# M_2} \leq \TC{M_2 \vee M_2}.\]
\end{Theorem}

Note that in \cite{Dranishnikov_Sadykov_2019}, Theorem \ref{thm:connected_sum_inequality} requires $\TC{M_i} \geq \frac{n+2}{r}$ for both $i=0$ and $i=1$, but by inspection of the proof only one of the $M_i$'s actually needs to satisfy the inequality.

\begin{Theorem}[\cite{Zapata_2021}]
\label{thm:wedge_equation}
Let $M$ and $N$ be closed manifolds. Then
\[ \TC{M \vee N} \leq \max \{ \TC{M}, \TC{N}, \cat{M \times N} \}.\]
\end{Theorem}

\subsection{Surgery}

Suppose that $S\subset M$ is an embedded $k$-sphere $S^k$ into an $n$-manifold $M$ that has a trivial tubular neighborhood $N=S^k\times D^{n-k}$. Thus, $N\subset\partial(D^{k+1}\times D^{n-k})$. The   $n$-manifold 
\[ M'=M\setminus Int(N)\cup (D^{k+1}\times\partial D^{n-k})\]
is said to be obtained from $M$ by {\em the $k$-surgery along the sphere $S$}. The {\em trace of this surgery} is defined to be the bordism $W$ between $M$ and $M'$ defined as 
\[ W = (M\times[0,1])\cup (D^{k+1}\times D^{n-k}).\]
By gluing together such bordisms, we can define the trace of a chain of surgeries. Moreover, we get the following homotopy equivalence of pairs that will later allow us to apply obstruction theory:

\begin{Proposition}[\cite{Crowley_Luck_Macko_2021,Dranishnikov_Scott_2020}]
\label{prop:surgery_trace}
There is a homotopy equivalence of pairs \[h:(M\cup_{\phi_1}D^{n_1}\cup_{\phi_2}D^{n_2}\cup\dots\cup_{\phi_k}D^{n_k},M)\to (W,M)\]
where  $W$ is the trace of a chain of surgeries in dimensions $n_1-1,\dots,n_k-1$ originating on a manifold $M$.
\end{Proposition}

One of the goals of surgery is to kill off parts of the homotopy groups (i.e., $\ker f_\ast$) of a manifold in order to change a map degree one $f$ into a homotopy equivalence, but to do so we need to add the additional technical condition that our map is normal:

\begin{Definition}[\cite{Browder_1972}]
Let $f:M \to N$ be a degree one map between smooth closed manifolds. Then $f$ is {\it normal} if there is a vector bundle $\nu$ on $N$ such that $\tau_M\oplus f^\ast\nu$ is a stably trivial bundle, where $\tau_M$ denotes the tangent bundle of $M$.
\end{Definition}

Moreover, the existence of a bordism of $f$ to a homotopy equivalence can be encoded algebraically:

\begin{Theorem}[The main theorem of surgery~\cite{Wall_1999}]
\label{thm:surgery_obstruction}
For every group $\pi$ there is a 4-periodic sequence of abelian groups $L_n(\pi)$ such that in dimension $n\ge 5$ for any degree one normal map $f:M\to N$ between closed manifolds there is an element
$\theta(f)\in L_n(\pi_1(N))$, called {\rm the surgery obstruction}, such that $f$
is (normally) bordant by means of a sequence of surgeries in dimensions
$\le\dim M/2$ to a (simple) homotopy equivalence if and only if
the surgery obstruction $\theta(f)\in L_n(\pi_1(N))$ is trivial. 
\end{Theorem}

One of the main assumptions of Theorem \ref{thm:general_Rudyak_conjecture} is that the surgery obstruction is trivial so that we can do surgery. Then in section 4 we work in the simply connected case where our L-groups are explicit and relatively simple, so we don't really work with general L-groups in this work.

\subsection{Simply-Connected Surgery}

If we assume $N$ is simply-connected, then the surgery obstructions of $f$ are much simpler:
\begin{enumerate}
    \item[\B] $L_{4k+1}(0) = 0$ so that $\theta(f) = 0$;
    \item[\B] $L_{4k+2}(0) = \mathbb{Z}_2$ and $\theta(f)$ is the Kervaire invariant of $f$, which is defined to be the Arf invariant of the intersection form of $f$ in $\mathbb{Z}_2$ coefficients;
    \item[\B] $L_{4k+3}(0) = 0$ so that $\theta(f) = 0$
    \item[\B] $L_{4k}(0) = \mathbb{Z}$ and $\theta(f) = \frac{1}{8} I(f) = \frac{1}{8} \left( I(M)-I(N) \right)$ where $I(f)$ is the signature of the intersection form of $f$ in $\mathbb{Z}$ coefficients.
\end{enumerate}

\begin{Theorem}[\cite{Browder_1972}] \label{thm:surgery_product_formula}
Let $f_i:M_i \to N_i$ for $i \in \{1,2\}$ be normal maps of degree one between smooth closed manifolds such that each $N_i$ is simply connected.
\begin{enumerate}
    \item[(i)] If $\dim\left(M_1 \times M_2 \right) \equiv 0 (\text{\rm mod } 4)$, then
    \[\theta(f_1 \times f_2) = I(N_1) \theta(f_2) + I(N_2) \theta(f_1) + 8 \theta(f_1) \theta(f_2).\]
    \item[(ii)] If $\dim\left(M_1 \times M_2 \right) \equiv 2 (\text{\rm mod } 4)$, then
    \[\theta(f_1 \times f_2) = \chi(N_1) \theta(f_2) + \chi(N_2) \theta(f_1).\]
\end{enumerate}
\end{Theorem}

\begin{Remark} Some notation and interpretations of Theorem \ref{thm:surgery_product_formula}:
\begin{enumerate}
    \item[\B] $\chi$ is the Euler characteristic modulo $2$.
    \item[\B] The signature of a manifold $I(M)$ is the signature of a normal map of degree $1$ of the form $M \to S^n$.
    \item[\B] If $\dim M_i$ is odd but $\dim(M_1 \times M_2)$ is even, we interpret $\theta(f_i)$ as $0$ sitting inside of the relevant obstruction group so that $\theta(f_1 \times f_2) = 0$.
    \item[\B] If $\dim M_i \equiv 2 (\text{mod }4)$ and $\dim(M_1 \times M_2) \equiv 0 (\text{mod }4)$, then $I(N_i) = 0$ and $8$ acting on the element $\theta(f_1)\theta(f_2)$ in $\mathbb{Z}_2$ gives $0$ so that $\theta(f_1 \times f_2) = 0$.
\end{enumerate}
\end{Remark}

Before the following lemma, note that $I(M \# N) = I(M) + I(N)$ and $I(\overline{M}) = -I(M)$ where $\overline{M}$ is the manifold $M$ taken with the opposite orientation. Both of these facts are also true for normal maps of degree one, but one must first justify that the connected sum of normal maps of degree one is well-defined, which we do in section 4.

\begin{Lemma}
\label{lem:closed_plumbing}
Let $f:M \to N$ be a normal map of degree one between smooth closed $4n$-manifolds with $N$ simply-connected. Then there is some $(2n-1)$-connected smooth closed manifold $P$ and a normal map of degree one $g:P \to S^{4n}$ with $\theta(g) = -\theta(f)$.
\end{Lemma}

\begin{proof}
Consider the degree one normal map $h:\overline{M} \# N \to S^{4n}$ given by collapsing $\left( \overline{M} \# N \right) \setminus D^{4n}$ to a point, where $D^{4n}$ is an embedded disk. Surgery can always be done below the middle dimension, i.e., $h$ is bordant to some $(2n)$-connected degree one normal map $g:P \to S^{4n}$. Since $g$ a is $(2n)$-connected map to a $(4n)$-sphere, the closed manifold in its domain, $P$, must be $(2n-1)$-connected. Moreover, since $g$ is bordant to $h$ so that
\[ \theta(g)
= \theta(h)
= \frac{1}{8}I(h)
= \frac{1}{8} \left( I(\overline{M} \# N) - I(S^{4k}) \right)\]
\[ = \frac{1}{8} \left( I(\overline{M}) + I(N) \right)
= \frac{1}{8} \left( I(N) - I(M) \right)
= -\theta(f) \qedhere\]
\end{proof}

\section{Rudyak's Conjecture for Sectional Category}

The following Theorem only requires minor edits in the proof of Lemma 3.2 given in \cite{Dranishnikov_Scott_2020}:

\begin{Theorem} \label{thm:general_Rudyak_conjecture}
Consider the following commutative diagram:
\[
\begin{tikzcd}
E^M \arrow[r, "\overline{f}"] \arrow[d, swap, "p^M"]
& E^N \arrow[d, "p^N"]
\\
M \arrow[r, "f"]
& N
\end{tikzcd}
\]
where $p^M$ and $p^N$ are fibrations, and $f$ is a normal map of degree $1$ between closed smooth manifolds. Let the fiber $F^N$ of $p^N$ be $(r-2)$-connected for some $r \geq 1$. If $f$ has zero surgery obstructions and $N$ satisfies $5 \leq \dim N \leq 2r \secat{p^N} - 3$, then $\secat{p^M} \geq \secat{p^N}$.
\end{Theorem}

\begin{proof}
By way of contradiction assume that $\secat{p^M} < \secat{p^N}$ and let $\secat{p^N} = q+1$. Let $\overline{f}_{q}: E^M_{q-1} \to E^N_{q}$ be the map induced by the above commutative diagram.
Since $\secat p^M \leq q$, there is a section $s^M:M \to E^M_{q}$ of $p^M_{q}$ so that the diagram
\[
\begin{tikzcd}
E^M_{q} \arrow[r, "\overline{f}"]
& E^N_{q} \arrow[d, "p_{q}^N"]
\\
M \arrow[r, swap, "f"] \arrow[u, "s^M"]
& N
\end{tikzcd}
\]
commutes. Let $\lambda = \overline{f} s^M$. Since $f$ has zero surgery obstruction, there is some normal bordism $F:W \to N$ of $f$ to some homotopy equivalence $f':M' \to N$ where the manifold $W$ is the trace of surgeries in dimensions $\leq \frac{\dim M}{2}$.
This gives us the following homotopy commutative diagram:
\[
\begin{tikzcd}
M \arrow[r,"\lambda"] \arrow[d, hook] & E^N_{q} \arrow[d, "p_{q}^N"]
\\ W \arrow[r, swap, "F"] \arrow[ur, dashed, "L"]
& N
\end{tikzcd}
\]
Our goal is to find a lift $L$ of $F$, labeled above, that extends $\lambda$ to $W$.
In view of Proposition \ref{prop:surgery_trace}, there is a homotopy equivalence of pairs
\[h:(M \cup_{\varphi_1}D^{n_1} \cup_{\varphi_2} ... \cup_{\varphi_k}D^{n_k},M)\to (W,M)\]
with  each $n_i \leq \frac{\dim M}{2} + 1$ when $\dim M$ is even and $n_i\leq\frac{\dim M-1}{2}+1$ when $\dim M$ is odd. So,  it suffices to solve the lifting problem 
\[
\begin{tikzcd}
W_{i-1} \arrow[r,"L_{i-1}"] \arrow[d, hook] & E^N_{q} \arrow[d, "p_{q}^N"]
\\ W_i \arrow[r, swap, "F_i"] \arrow[ur, dashed, "L_i"]
& N
\end{tikzcd}
\]
for each $i>0$ where $W_i=M \cup_{\varphi_1}D^{n_1} \cup_{\varphi_2} ... \cup_{\varphi_i}D^{n_i}$, $W_0=M$, $L_0=\lambda$, and $F_i$ is the restriction of $F\circ h$ to $W_i$. 

We recall that the fiber of $p_{q}^N$ is $\ast^{q+1} F^N$, which is $(r(q+1)-2)$-connected since $F^N$ is $(r-2)$-connected. Thus,  $L_{i-1}$ can be extended to $D^{n_i}$ if $n_i - 1 \leq r(q+1) - 2$. Suppose that $\dim M$ is even. Then $\dim M$ and $2r(q+1)-3$ differ by at least $1$ so that our assumption strengthens to $n \leq 2r(q+1)-4$ and the lift $L_i$ exists since $n_i - 1 \leq \frac{\dim M}{2}$. If $\dim M$ is odd, then $n_i-1 \leq 2r(q+1)-3$ by the assumption and again the lift $L_i$ still exists. 

Thus, the lift $L$ exists. Now since $f'$ is a homotopy equivalence, it has some homotopy inverse $g'$ so that $L \circ g'$ is a homotopy section of $p^N_{q}$, but this contradicts the fact that $\secat{p^N} = q+1$.  Hence, $\secat{p^M} \geq \secat{p^N}$.
\end{proof}

\begin{Corollary}\label{cor:general_TCk_Rudyak}
Let $f:M \to N$ be a normal map of degree $1$ between closed smooth manifolds such that $N$ is $(r-1)$-connected for some $r \geq 1$. If $f^{\times k}$ has no surgery obstructions and $N$ satisfies the inequality $5 \leq k \dim N \leq 2r \higherTC{k}{N} - 3$, then $\higherTC{k}{M} \geq \higherTC{k}{N}$
\end{Corollary}

\section{The Simply-Connected Case for Higher TC}

Unfortunately, we can no longer follow the methodology of \cite{Dranishnikov_Scott_2020} because a formula for $\higherTC{k}{M \# M'}$ is unknown for any $k$, even with the assumption that $M=M'$. Instead, we consider the case where the codomain $N$ is simply connected since surgery is vastly simpler with this assumption. In this case, we're able to get the exact result we want in three cases out of four.

\begin{Theorem} \label{thm:simply-connected_TCk_Rudyak}
Let $f:M \to N$ be a normal map of degree $1$ between closed smooth manifolds such that $N$ is $(r-1)$-connected for some $r \geq 2$ (i.e., $N$ is simply connected). Suppose that $N$ satisfies the inequality $5 \leq k \dim N \leq 2r \higherTC{k}{N} - 3$. If $\dim N \not\equiv 0 (\text{\rm mod} 4)$, then $\higherTC{k}{M} \geq \higherTC{k}{N}$.
\end{Theorem}

\begin{proof}
First we assume $\dim N$ is odd. In this case, $\theta(f^{\times k}) = 0$ for all $k \geq 1$ by Theorem \ref{thm:surgery_product_formula} so that $\higherTC{k}{M} \geq \higherTC{k}{N}$ by Corollary \ref{cor:general_TCk_Rudyak}.

Now assume that $\dim N \equiv 2 (\text{\rm mod } 4)$. Then $\theta(f^{\times k}) = 0$ for all $k$ even by Theorem \ref{thm:surgery_product_formula}. Now suppose $k \geq 3$ is odd so that $\dim N^k \equiv 2 (\text{\rm mod } 4)$. Then
\[ \theta(f^{\times k})
= \chi(N) \theta(f^{k-1}) + \chi(N^{k-1}) \theta(f)
= \chi(N)^{k-1} \theta(f).\]
Since $\dim N \equiv 2 (\text{\rm mod } 4)$, it follows that $\chi(N)$ is even (see page 252 of \cite{Hatcher_2001}) so that $\theta(f^{\times k}) = 0$ in $\mathbb{Z}_2$.
\end{proof}

In order to prove anything about the $4k$-dimensional case, we first need to be able to take connected sums of normal maps of degree one $f:M \to N$. In \cite{Dranishnikov_Scott_2020}, this was done by changing the manifold $M$ in the domain of $f$ via surgery without changing its LS-category. It's difficult to compute the TC after surgery, so instead we prove Lemma \ref{lem:trivial_preimage} and Proposition \ref{prop:connected_sum_maps} below to show that taking connected sums of normal maps of degree one maps is well-defined up to homotopy.

The proofs of Lemma \ref{lem:trivial_preimage} and Proposition \ref{prop:connected_sum_maps} are extracted from the proof of Lemma 13.42 of \cite{Crowley_Luck_Macko_2021}:

\begin{Lemma}
\label{lem:trivial_preimage}
Let $f:\mathbb{R}^n \to \mathbb{R}^n$ be a smooth map such that $0 \in \mathbb{R}^n$ is a regular value of $f$ with preimage $f^{-1}(0) = \{ x_0,x_1 \}$ contained in the interior of the disk $D^n \subset \mathbb{R}^n$. If the differentials $T_{x_0}f:T_{x_0}\mathbb{R}^n \to T_0 \mathbb{R}^n$ and $T_{x_1}f:T_{x_1}\mathbb{R}^n \to T_0 \mathbb{R}^n$ are orientation preserving and orientation reversing, respectively, then $f$ can be changed up to homotopy relative $\mathbb{R}^n \setminus D^n$ such that $f^{-1}(0)$ is empty.
\end{Lemma}

\begin{proof}
First let $\varepsilon >0$ be small enough such that $f(S^{n-1})$ is disjoint from $\varepsilon D^{n}$. Let $r_\varepsilon:\mathbb{R}^n \to \varepsilon D^n$ be the retraction given by
\[ r_\varepsilon(x) =
\begin{cases}
x & \text{if } \left\| x \right\| \leq \varepsilon \\
\varepsilon \frac{x}{\left\| x \right\|} & \text{if } \left\| x \right\| \geq \varepsilon
\end{cases}.\]
Then $r_\varepsilon \circ f$ induces a map $(D^n,S^{n-1}) \to (\varepsilon D^n, \varepsilon S^{n-1})$ of compact oriented manifolds. Since $T_{x_0}f:T_{x_0}\mathbb{R}^n \to T_0 \mathbb{R}^n$ and $T_{x_1}f:T_{x_1}\mathbb{R}^n \to T_0 \mathbb{R}^n$ have opposite degrees, it follows that $0 \in \varepsilon D^n$ has degree 0; hence, the induced map $S^{n-1} \to \varepsilon S^{n-1}$ has degree 0 and in nullhomotopic. Therefore, the map $f_0:S^{n-1} \to \mathbb{R} \setminus \{ 0 \}$ induced by $f$ is also nullhomotopic so that is extends to a map $f_1:D^n \to \mathbb{R}^n \setminus \{ 0 \}$. Now define $f':\mathbb{R}^n \to \mathbb{R}^n \setminus 0$ such that $f'|_{\mathbb{R}^n \setminus D^n} = f|_{\mathbb{R}^n \setminus D^n}$ and $f'|_{D^n} = f_1$. Note that the mapping
\[ (x,t) \mapsto tf'(x) + (1-t)f(x) \]
defines a homotopy from $f'$ to $f$. Furthermore, $0$ is not in the image of $f'$ so that the result follows.
\end{proof}

\begin{Proposition}
\label{prop:connected_sum_maps}
Let $f:M \to N$ be a map of degree 1 with between smooth closed $n$-manifolds such that $N$ is simply-connected. Then $f$ is homotopic to some map $g:M \to N$ such that there is some disk $D^n \subset N$ with $g|_{g^{-1}(D^{n})}:g^{-1}(D^{n}) \to D^{n}$ a diffeomorphism. Moreover, if $f$ is normal then $g$ is normal.
\end{Proposition}

\begin{proof}
We first choose orientations on $M$ and $N$. Since the cardinality of the fiber of a regular value of smooth degree one map is locally constant, it suffices to homotope $f$ such that some regular value has a fiber consisting of a single point. Since regular values have full measure in the codomain, it follows that $f$ has some regular value $y \in N$. Since $y$ is a regular value, its fiber is finite. Moreover,
\[ 1 = \deg f = \sum_{x \in f^{-1}(y)} \deg(f,x) \]
where $\deg(f,x) = 1$ if $T_xf:T_xM \to T_yN$ is orientation preserving and $\deg(f,x) = -1$ if $T_xf$ is orientation reversing. Thus, we can enumerate $f^{-1}(y) = \{ x_0,...,x_{2m}\}$ such that $\deg(f,x_i) = 1$ if $i$ is even and $\deg(f,x) = -1$ if $i$ is odd. All that remains is to describe a procedure that homotopes $f$ to some $g$ such that $g^{-1}(y) = \{ x_0,...,x_{2m-2} \}$ since such a process can simply be repeated till $g^{-1}(y) = \{x_0\}$.

 Let $\alpha$ be an embedded arc in $M$ joining $x_{2m-1}$ with $x_{2m}$, and let $\text{Tub}(\alpha)$ be a tubular neighborhood of $\alpha$ inside of $M$. Let $V$ be an open neighborhood of $y$ diffeomorphic to $\mathbb{R}^n$. Since $N$ is simply-connected, we can choose a homotopy $r_t:N \to N$ rel $V$ such that $r_1 \circ f(\text{Tub}) \subset V$. Thus, if we take the restriction $r_1 \circ f|_{\text{Tub}(\alpha)}:\text{Tub}(\alpha) \simeq \mathbb{R}^n \to V \simeq \mathbb{R}^n$, then we are in the case of Lemma \ref{lem:trivial_preimage}. Therefore, up to homotopy we can remove $x_{2m-1}$ and $x_{2m}$ from the preimage of $y$, i.e., $f$ is homotopic to some $g$ such that $g^{-1}(y) = \{ x_0,...,x_{2m-2} \}$. Moreover, since homotopies lift to bundle maps, the normality of $f$ would imply the normality of $g$.
\end{proof}

\begin{Remark}
Proposition \ref{prop:connected_sum_maps} shows that taking connected sums of normal maps of degree 1 is, up to homotopy, a well-defined operation when the codomain is simply connected.
\end{Remark}

\begin{Theorem}
Let $f:M \to N$ be a normal map of degree one between $(r-1)$-connected ($r \geq 2$) smooth closed $4n$-manifolds, and assume that $M$ satisfies the inequality
\[\max \left\{ \cat{M} + 2, \frac{4n+2}{r} \right\} \leq \TC{M}.\]
Then $\TC{M} \geq \TC{N}$.
\end{Theorem}

\begin{proof}
Note that if $\cat{M} = 1$, then $M$ is a sphere so that $f$ has to be a homotopy equivalence and $\cat{M} = \cat{N}$. Therefore, we can assume $\cat{M} \geq 2$.

Now we assume by way of contradiction that $\TC{N} > \TC{M}$. Then we get that
\[ \frac{4n+2}{r} \leq \TC{M} < \TC{N} \]
so that rearranging gives us the inequality
\[ 8n \leq r\TC{N}-4 \leq r\TC{N}-3.\]
If $n=0$, then $\TC{M} = \TC{N} = 0$ so we can assume $n>1$ to retrieve the inequality $5 \leq 8n$ as well. Now we simply need to construct the normal map of degree one on which to apply Theorem \ref{thm:general_Rudyak_conjecture}. By Lemma \ref{lem:closed_plumbing}, there is some $(2n-1)$-connected smooth closed manifold $P$ and some normal map of degree one $g:P \to S^{4n}$ such that $\theta(g) = -\theta(f)$. Then by Proposition \ref{prop:connected_sum_maps}, $f \# g:M \# P \to N$ is a normal map of degree one. Moreover, $\theta(f \# g) = 0$ since $\theta(g) = - \theta(f)$. Therefore, by Theorem \ref{thm:general_Rudyak_conjecture} we get $\TC{M \# P} \geq \TC{N}$.

Since $\frac{4n+2}{r} \leq \TC{M}$, we can apply Theorem \ref{thm:connected_sum_inequality} to get $\TC{M \# P} \leq \TC{M \vee P}$. Moreover, Theorem \ref{thm:wedge_equation} gives us an equation for $\TC{M \vee P}$ so that we have the inequality
\[ \TC{N} \leq \TC{M \vee P} = \max \{ \TC{M}, \TC{P}, \cat{M \times P} \}.\]
Since $\cat{P} \leq 2$ and $\TC{P} \leq 4$, it follows that
\[ \TC{N} \leq \max \{ \TC{M}, 4, \cat{M} + 2 \}\]
Note that $\cat{M} + 2 \leq \TC{M}$ by assumption. Furthermore, $\cat{M} \geq 2$ (see beginning of proof) so that $4 \leq \TC{M}$. Therefore,
\[\TC{N} \leq \TC{M}. \qedhere\]
\end{proof}

\section*{Acknowledgements}
I would like to thank my advisor, Alexander Dranishnikov, for all of his help and encouragement throughout this project.

\footnotesize
\bibliographystyle{plain}
\bibliography{main}

\end{document}